\documentclass{birkjour}
\usepackage{amsmath,amsthm,amssymb,mathtools,enumitem,setspace}

\usepackage[pagebackref,colorlinks=true,citecolor=blue,linkcolor=red,urlcolor=blue]{hyperref}

 \newtheorem{thm}{Theorem}[section]

 \newtheorem{prop}[thm]{Proposition}
 \theoremstyle{definition}
 
 \theoremstyle{remark}
 \newtheorem{rem}[thm]{Remark}
 
 \numberwithin{equation}{section}

\title[The diameter of the intersection graph of a finite simple group]{The intersection graph of a finite\\simple group has diameter at most 5}
\author{Saul D. Freedman}
\date{\today}
\address{ %
School of Mathematics and Statistics\\ University of St Andrews\\ St Andrews\\ KY16 9SS, UK}
\email{sdf8@st-andrews.ac.uk}

\newcommand\ig{\Delta_G}
\newcommand\diam{\mathrm{diam}}
\newcommand\sd{\mkern1.5mu{:}\mkern1.5mu}
\newcommand{\nonsplit}[2]{#1\raisebox{0.6ex}{$\cdot$} #2}

\renewcommand{\le}{\leqslant}
\renewcommand{\ge}{\geqslant}

\begin{document}

\onehalfspacing

\begin{abstract}
Let $G$ be a non-abelian finite simple group. In addition, let $\ig$ be the intersection graph of $G$, whose vertices are the proper nontrivial subgroups of $G$, with distinct subgroups joined by an edge if and only if they intersect nontrivially. We prove that the diameter of $\ig$ has a tight upper bound of $5$, thereby resolving a question posed by Shen (2010). Furthermore, a diameter of $5$ is achieved only by the baby monster group and certain unitary groups of odd prime dimension.
%Let $G$ be a finite non-abelian simple group. In addition, let $\ig$ be the intersection graph of $G$, whose vertices are the proper nontrivial subgroups of $G$, with distinct subgroups joined by an edge if and only if they intersect nontrivially. We prove that $\ig$ has diameter at most $5$, and that a diameter of $5$ is achieved by the baby monster group $\mathbb{B}$, thereby resolving a question posed by Shen (2010). Furthermore, if $\ig$ has diameter $5$ and $G \not\cong \mathbb{B}$, then $G$ is a unitary group of odd prime dimension. Some, but not all, of these groups achieve a diameter of $5$.
\end{abstract}

\thanks{The author was supported by a St Leonard's International Doctoral Fees Scholarship and a School of Mathematics \& Statistics PhD Funding Scholarship at the University of St Andrews.}

\maketitle

\section{Introduction}
\label{sec:intro}

For a finite group $G$, let $\ig$ be the \emph{intersection graph} of $G$. This is the graph whose vertices are the proper nontrivial subgroups of $G$, with two distinct vertices $S_1$ and $S_2$ joined by an edge if and only if $S_1 \cap S_2 \ne 1$. We write $d(S_1,S_2)$ to denote the distance in $\ig$ between vertices $S_1$ and $S_2$, and if these vertices are joined by an edge, then we write $S_1 \sim S_2$. Additionally, $\diam(\ig)$ denotes the diameter of $\ig$.

Cs\'{a}k\'{a}ny and Poll\'{a}k \cite{csakany} introduced the graph $\ig$ in 1969, as an analogue of the intersection graph of a semigroup defined by Bos\'{a}k \cite{bosak} in 1964. For finite non-simple groups $G$, Cs\'{a}k\'{a}ny and Poll\'{a}k determined the cases where $\ig$ is connected, and proved that, in these cases, $\diam(\ig) \le  4$ (see also \cite[Lemma 5]{shen}). It is not known if there exists a finite non-simple group $G$ with $\diam(\ig) = 4$.

Suppose now that $G$ is a non-abelian finite simple group. In 2010, Shen \cite{shen} proved that $\ig$ is connected, and asked two questions: does $\diam(\ig)$ have an upper bound? If yes, does the upper bound of $4$ from the non-simple case also apply here? In the same year, Herzog, Longobardi and Maj \cite{herzog} independently showed that the subgraph of $\ig$ induced by the \emph{maximal} subgroups of $G$ is connected with diameter at most $62$. As each proper nontrivial subgroup of $G$ is adjacent in $\ig$ to some maximal subgroup, this implies an upper bound of $64$ for $\diam(\ig)$, resolving Shen's first question. Ma \cite{ma} reduced this upper bound to 28 in 2016. In the other direction, Shahsavari and Khosravi \cite[Theorem 3.7]{shahsavari} proved in 2017 that $\diam(\ig) \ge 3$. %In 2020, these same authors \cite{shahsavari2020} determined all cases in which $\ig$ has a leaf, and showed in these cases that if $H$ is a finite group satisfying $\Delta_H \cong \ig$, then $H \cong G$. Finally, Ahmadi and Taeri \cite[Theorem 4.2]{ahmadi} proved in 2016 that $\ig$ is non-planar.

In this paper, we significantly reduce the previously known upper bound of $28$ for $\diam(\ig)$, and show that the new bound is best possible. In particular, we prove the following theorem, which resolves Shen's second question with a negative answer.

\begin{thm}
\label{thm:diamthm}
Let $G$ be a non-abelian finite simple group.
\begin{enumerate}[label={(\roman*)},font=\upshape]
\item $\ig$ is connected with diameter at most $5$.
\item \label{diamthm2} If $G$ is the baby monster group $\mathbb{B}$, then $\diam(\ig) = 5$. 
\item If $\diam(\ig) = 5$ and $G \not\cong \mathbb{B}$, then $G$ is a unitary group $\mathrm{U}_n(q)$, with $n$ an odd prime and $q$ a prime power.
\end{enumerate}
\end{thm}

\begin{rem}
Using information from the \textsc{Atlas} \cite{ATLAS}, we can show that if $S_1$ and $S_2$ are vertices of $\Delta_{\mathbb{B}}$ with $d(S_1,S_2) = 5$, then $|S_1| = |S_2| = 47$.
\end{rem}

\begin{rem}
If $G \in \{\mathrm{U}_3(3), \mathrm{U}_3(5),\mathrm{U}_5(2)\}$, then $G$ has no maximal subgroup of odd order \cite[Theorem 2]{liebecksaxl}. As we will explain in the proof of Theorem \ref{thm:diamthm}, this implies that $\diam(\ig) \le 4$. Indeed, we can use information from the \textsc{Atlas} \cite{ATLAS} to show that $\diam(\Delta_{\mathrm{U}_3(3)}) = 3$. Furthermore, even though $\mathrm{U}_3(7)$ has a maximal subgroup of odd order, we deduce from calculations in Magma \cite{magma} that $\diam(\Delta_{\mathrm{U}_3(7)}) = 4$. On the other hand, we can adapt the proof of Theorem \ref{thm:diamthm}\ref{diamthm2}, with the aid of several Magma calculations, to show that $\diam(\Delta_{\mathrm{U}_7(2)}) = 5$.
\end{rem}

It is an open problem to classify the finite simple unitary groups $G$ with $\diam(\ig) = 5$.

\section{Proof of Theorem \ref{thm:diamthm}}
\label{sec:upperbound}

In order to prove Theorem \ref{thm:diamthm} in the unitary case, we will require the following proposition. For a prime power $q$, let $f$ be the unitary form on the vector space $V:=\mathbb{F}_{\!q^2}^3$ whose Gram matrix is the $3 \times 3$ identity matrix, and let $\mathrm{SU}_3(q)$ be the associated special unitary group. Then the standard basis for $(V,f)$ is orthonormal, and a matrix $A \in \mathrm{SL}_3(q^2)$ lies in $\mathrm{SU}_3(q)$ if and only if $A^{-1} = A^{\sigma \mathsf{T}}$, where $\sigma$ is the field automorphism $\alpha \mapsto \alpha^q$ of $\mathbb{F}_{\!q^2}$. For a subspace $U$ of $V$, we will write $\mathrm{SU}_3(q)_U$ to denote the stabiliser of $U$ in $\mathrm{SU}_3(q)$.

\begin{prop}
\label{prop:interst}
Let $q$ be a prime power greater than $2$, and let $X$ and $Y$ be one-dimensional subspaces of the unitary space $(V,f)$, with $X$ non-degenerate. Then $\mathrm{SU}_3(q)_X \cap \mathrm{SU}_3(q)_Y$ contains a non-scalar matrix.
\end{prop}

\begin{proof}
We may assume without loss of generality that $X$ contains the vector $(1,0,0)$. Let $(a,b,c)$ be a nonzero vector of $Y$. In addition, let $\omega$ be a primitive element of $\mathbb{F}_{\!q^2}$, and let $\lambda:=\omega^{q-1}$. Then $|\lambda| = q+1 > 3$. If at least one of $a$, $b$ and $c$ is equal to $0$, then $\mathrm{SU}_3(q)_X \cap \mathrm{SU}_3(q)_Y$ contains a non-scalar diagonal matrix with two diagonal entries equal to $\lambda$ and one equal to $\lambda^{-2}$ (not necessarily in that order).

Suppose now that $a$, $b$ and $c$ are all nonzero, and let $\mu:=b^{-1}c$. We may assume that $a = 1$. The trace map ${\alpha \mapsto \alpha + \alpha^q}$ from $\mathbb{F}_{\!q^2}$ to $\mathbb{F}_{\!q}$ is $\mathbb{F}_{\!q}$-linear, and hence has a nontrivial kernel. In particular, there exists $\beta \in \mathbb{F}_{\!q^2}$ such that $\beta \ne 1$ and $\beta + \beta^q = 2$. It follows from simple calculations that if $\mu^{q+1} = -1$, then $\mathrm{SU}_3(q)_X \cap \mathrm{SU}_3(q)_Y$ contains
\begin{align*}
\left( \begin{matrix}
1&0&0\\
0&\beta&\mu(1-\beta^q)\\
0&\mu^{-1}(1-\beta)&\beta^q
\end{matrix} \right).
\end{align*}
If instead $\mu^{q+1} \ne -1$, then we can define $\gamma:=\lambda^{-2}(\lambda^3+\mu^{q+1})(1+\mu^{q+1})^{-1}$. In this case, $\mathrm{SU}_3(q)_X \cap \mathrm{SU}_3(q)_Y$ contains
\begin{align*}
\left( \begin{matrix}
\lambda&0&0\\
0&\gamma&\mu(\lambda-(\gamma \lambda)^q)\\
0&\mu^{-1}(\lambda-\gamma)&(\lambda \gamma)^q
\end{matrix} \right).
\end{align*}
Note that $\lambda \ne \gamma$, since $|\lambda| > 3$.
\end{proof}

\begin{proof}[Proof of Theorem \ref{thm:diamthm}]
%Since each proper nontrivial subgroup of $G$ contains a group of prime order, $\diam(\ig)$ is the maximum distance in $\ig$ between subgroups of prime order, as noted in the proof of \cite[Theorem 2]{csakany}.
Let $S_1$ and $S_2$ be proper nontrivial subgroups of $G$, and let $M_1$ and $M_2$ be maximal subgroups of $G$ that contain $S_1$ and $S_2$, respectively. Since $d(M_1,M_2) \le d(S_1,M_2) \le d(S_1,S_2)$, we may assume that $S_1$ and $S_2$ are not maximal in $G$. We may also assume that $M_1 \ne M_2$, as otherwise $S_1 \sim M_1 \sim S_2$ and $d(S_1,S_2) \le 2$.

Suppose first that $|M_1|$ and $|M_2|$ are even. Then, as observed in the proof of \cite[Proposition 3.1]{herzog}, there exist involutions $x \in M_1$ and $y \in M_2$, with $\langle x,y \rangle$ equal to a (proper) dihedral subgroup $D$ of $G$ (with $|D| = 2$ allowed). Hence $S_1 \sim M_1 \sim D \sim M_2 \sim S_2$, and so $d(S_1,S_2) \le 4$. In particular, if every maximal subgroup of $G$ has even order, then $\diam(\ig) \le 4$, as noted in the proof of \cite[Lemma 2.3]{ma}.

It remains to consider the case where $G$ contains a maximal subgroup of odd order. Liebeck and Saxl \cite[Theorem 2]{liebecksaxl} present a list containing all possibilities for $G$ and its maximal subgroups of odd order. By the previous paragraph, we may assume that the maximal subgroup $M_1$ has odd order. However, $|M_2|$ may be even. In what follows, information about the sporadic simple groups is taken from the \textsc{Atlas} \cite{ATLAS}, except where specified otherwise. %In fact, each listed group does contain a maximal subgroup of odd order, except for the monster group. In this case, the listed subgroup of shape $59 \sd 29$ lies in the maximal subgroup $\mathrm{L}_2(59)$ constructed in \cite{holmes59}, while the subgroup of shape $71 \sd 35$ lies in the maximal subgroup $\mathrm{L}_2(71)$ constructed in \cite{holmes71}. Note that the shapes of these groups, as well as all information about the sporadic simple groups in the remainder of this proof, are taken from the \textsc{Atlas} \cite{ATLAS}.

%We now investigate each remaining case from Liebeck and Saxl's list. In particular, we will show that $\diam(\Delta_{\mathbb{B}}) = 5$, that $d(S_1,S_2) \le 5$ if $G$ is a linear or unitary group specified in the statement of the theorem, and that $d(S_1,S_2) \le 4$ otherwise.

\medskip

%\begin{itemize}[leftmargin=*]
\noindent (i) $G = A_p$, with $p$ prime, $p \equiv 3 \pmod 4$, and $p \notin \{7,11,23\}$. By \cite[Theorem 2]{csakany} (see also \cite[Assertion I]{shen}), the intersection graph of any simple alternating group has diameter at most $4$.

\medskip

\noindent (ii) $G = \mathrm{L}_2(q)$, with $q$ a prime power and $q \equiv 3 \pmod 4$. %If $|M_2|$ is odd, then $M_1 \sim M_2$ \cite[p.~323]{herzog}, and so $d(S_1,S_2) \le 3$. We will therefore assume that $|M_2|$ is even. The following argument is due to Peter Cameron.
The group $G$ acts transitively on the set $\Omega$ of one-dimensional subspaces of the vector space $\mathbb{F}_{\!q}^2$. Additionally, $M_1 = G_U$ for some $U \in \Omega$, and $G_U \cap G_W \ne 1$ for each $W \in \Omega$. If $|M_2|$ is odd, then $M_2 = G_W$ for some $W$, and it follows that $M_1 \sim M_2$ and $d(S_1,S_2) \le 3$. We may therefore assume that $M_2$ contains an involution $g$. Then $g$ fixes no subspace in $\Omega$, and so $g \in G_{\{U,X\}}$ for some $X \in \Omega \setminus \{U\}$. Since the nontrivial subgroup $G_U \cap G_X$ lies in both $M_1 = G_U$ and $G_{\{U,X\}}$, we deduce that $S_1 \sim M_1 \sim G_{\{U,X\}} \sim M_2 \sim S_2$. Thus $d(S_1,S_2) \le 4$.
%
%
% If $K_1$ and $K_2$ are maximal subgroups of $G$, then $d(K_1,K_2) \le 3$ \cite[Proposition 3.3]{herzog}. Hence $d(S_1,S_2) \le 5$.
%
%Suppose now that $q$ is a proper power of a prime $p$. Each subgroup of $G$ whose order is a prime not equal to $p$ lies in a maximal subgroup of even order \cite[p.~323]{herzog}. Additionally, it is well known that $G$ has a unique conjugacy class of subgroups of order $p$. Therefore, every subgroup of order $p$ lies in a conjugate of the proper subgroup $\mathrm{L}_2(p)$ of $G$. Hence each subgroup of $G$ of prime order lies in a maximal subgroup of even order, and so $\diam(\ig) \le 4$.
% we deduce from the character table of $G$ (see \cite[Table 2]{fritzsche})

%Additionally, the Sylow $p$-subgroup $H$ of $G$ is elementary abelian, and $N_G(H)$ acts transitively on the subgroups of order $p$ in $H$ (see, for example, \cite[Theorem 2.8.2]{gorenstein}). Hence $G$ has a unique conjugacy class of subgroups of order $p$. Therefore, every subgroup of order $p$ lies in a conjugate of the proper subgroup $\mathrm{L}_2(p)$ of $G$. It follows that each subgroup of $G$ of prime order lies in a maximal subgroup of even order, and so $\diam(\ig) \le 4$.

\medskip

\noindent (iii) $G = \mathrm{L}_n(q)$, with $n$ an odd prime, $q$ a prime power, and $G \not\cong \mathrm{L}_3(4)$. Similarly to the previous case, the group $G$ and its overgroup $R:=\mathrm{PGL}_n(q)$ act transitively on the set $\Omega$ of one-dimensional subspaces of the vector space $\mathbb{F}_{\!q}^n$. Here, $M_1 = {G \cap N_R(K)}$, where $K$ is a Singer subgroup of $R$, i.e., a cyclic subgroup of order $(q^n-1)/(q-1)$ (see \cite[\S1--2]{hestenes}).

Now, $M_1$ contains a non-identity element $m$ that fixes a subspace $X \in \Omega$ \cite[p.~497]{hestenes}. Observe that ${m^k \in M_1}$ for each $k \in K$. The action of $K$ on $\Omega$ is transitive, and hence each subspace in $\Omega$ is fixed by some non-identity element of $M_1$. Therefore, if a non-identity element of $S_2$ fixes a subspace $U \in \Omega$, then $S_1 \sim M_1 \sim G_U \sim S_2$ and $d(S_1,S_2) \le 3$. Otherwise, since $n$ is prime, there exists $g \in G$ such that $S_2 \cap M_1^g \ne 1$. Thus ${S_1 \sim M_1 \sim G_X \sim M_1^g \sim S_2}$ and $d(S_1,S_2) \le 4$.

\medskip

\noindent (iv) $G = \mathrm{U}_n(q)$, with $n$ an odd prime, $q$ a prime power, and $G \not\cong \mathrm{U}_3(3)$, $\mathrm{U}_3(5)$ or $\mathrm{U}_5(2)$. Here, $G$ acts intransitively on the set of one-dimensional subspaces of the vector space $\mathbb{F}_{\!q^2}^n$. Let $(q+1,n)$ denote the greatest common divisor of $q+1$ and $n$. The maximal subgroup $M_1$ is equal to $N_G(T)$, where $T$ is a Singer subgroup of $G$, i.e., a cyclic subgroup of order $\frac{q^n+1}{(q+1)(q+1,n)}$ (see \cite[\S5]{hestenes}). In fact, each maximal subgroup of $G$ of odd order is conjugate to $M_1$. Similarly to the linear case, $M_1$ contains a non-identity element that fixes a one-dimensional subspace $X$ of $\mathbb{F}_{\!q^2}^n$ \cite[p.~512]{hestenes}.
%With respect to a normal basis for the vector space $\mathbb{F}_{\!q^u}^n$, this automorphism can be represented as the $n \times n$ matrix 
%\begin{align*}B:=
%\left( \begin{matrix}
%0&1&0&\cdots&0\\
%0&0&1&\cdots&0\\
%\vdots&\vdots&\vdots&\ddots&\vdots\\
%0&0&0&\cdots&1\\
%1&0&0&\cdots&0
%\end{matrix} \right).
%\end{align*}
%Since $n$ is odd, $B \in H$. Additionally, in the unitary case, two $\mathrm{GL}_n(q^2)$-conjugate elements of $H$ are also $\mathrm{GU}_n(q^2)$-conjugate \cite[p.~34]{wall}. We may therefore assume, in either case, that $B \in \hat{M_1}$.

Let $L:=G_X$. Then $M_1 \sim L$, and we can calculate $|L|$ using \cite[Table 2.3]{bhrd}. In particular, $|L|$ is even. Hence if $|M_2|$ is even, then $G$ contains a dihedral subgroup $D$ such that ${S_1 \sim M_1 \sim L \sim D \sim M_2 \sim S_2}$, and $d(S_1,S_2) \le 5$. If $|M_2|$ is odd, then there exists an element $g \in G$ such that $M_2 = M_1^g$. Thus $L^g \sim M_2$. If $n = 3$ and $X$ is non-degenerate, then it follows from Proposition \ref{prop:interst} that $L \sim L^g$. Therefore, ${S_1 \sim M_1 \sim L \sim L^g \sim M_2 \sim S_2}$ and $d(S_1,S_2) \le 5$. In the remaining cases, we will show that $|L|^2/|G| > 1$, and hence $|L|\,|L^g| > |G|$. It will follow that $L \cap L^g \ne 1$, again yielding $d(S_1,S_2) \le 5$.

Observe that $|L|^2/|G| > 1$ if and only if $\log |G|/\log |G:L| > 2$. By \cite[Proposition 3.2]{halasi}, if $n \ge 7$, then $\log |G|/\log |G:L| > 2$, as required. If instead $n = 3$, then we may assume that $X$ is totally singular. Here, $q > 2$, and hence $$|L|^2/|G| = \frac{q^3(q^2-1)}{(q^3+1)(q+1,3)} \ge \frac{q^3(q-1)}{(q^3+1)} > 1.$$ Suppose finally that $n = 5$. If $X$ is totally singular, then $|L|^2/|G|$ is equal to $$\frac{q^{10}(q^2-1)^3(q^3+1)}{(q^4-1)(q^5+1)(q+1,5)} > \frac{q^{10}}{(q^4-1)(q^5+1)} = \frac{q^{10}}{q^9-q^5+q^4-1} > 1.$$ If instead $X$ is non-degenerate, then $$|L|^2/|G| = \frac{q^2(q+1)\prod_{i=1}^{4}(q^i-(-1)^i)}{(q^5+1)(q+1,5)} > \frac{q^2(q^{4}-1)}{q^5+1} = \frac{q^6-q^2}{q^5+1} > 1.$$

\medskip

\noindent (v) $G = \mathrm{M}_{23}$. In this case, $M_1$ has shape $23 \sd 11$. We argue as in the proof of \cite[Assertion I]{shen}. There exists a maximal subgroup $L$ of $G$ isomorphic to $\mathrm{M}_{22}$, and $|M_1|\,|L|$ and $|M_2|\,|L|$ are greater than $|G|$ (for any choice of $M_2$). It follows that $S_1 \sim M_1 \sim L \sim M_2 \sim S_2$, and so $d(S_1,S_2) \le 4$.

\medskip

\noindent (vi) $G = \mathrm{Th}$. Here, $M_1$ has shape $31 \sd 15$. If the proper nontrivial subgroup $S_1$ of $M_1$ has order $31$, then $S_1$ lies in a maximal subgroup of shape $\nonsplit{2^5}{\mathrm{L}_5(2)}$. Otherwise, $|C_G(S_1)|$ is even. Therefore, in each case, $S_1$ lies in a maximal subgroup of even order. The same is true for $S_2$, and thus $d(S_1,S_2) \le 4$.

\medskip

\noindent (vii) $G = \mathbb{B}$. In this case, $M_1$ has shape $47 \sd 23$. Additionally, $G$ has a maximal subgroup $K \cong \mathrm{Fi}_{23}$, which has even order, and $M_1 \sim K$. Hence if $|M_2|$ is even, then $S_1 \sim M_1 \sim K \sim D \sim M_2 \sim S_2$ for some dihedral subgroup $D$ of $G$, yielding $d(S_1,S_2) \le 5$. Otherwise, there exists an element $g \in G$ such that $M_2 = M_1^g$, and hence $K^g \sim M_2$. As $|K|^2/|G| > 1$, we conclude that ${S_1 \sim M_1 \sim K \sim K^g \sim M_2 \sim S_2}$ and $d(S_1,S_2) \le 5$. Thus $\diam(\ig) \le 5$.

We now show that $\diam(\ig)$ is equal to $5$. Let $H$ be a subgroup of $M_1$ of order $23$. Then $H$ is a Sylow subgroup of $G$. It follows from \cite[p.~67]{wilsonsporadic} that each maximal subgroup of $G$ that contains $H$ is conjugate either to $M_1$, to $K$, or to a subgroup $L$ of shape $\nonsplit{2^{1+22}}{\mathrm{Co}_2}$. We may assume that $H \le M_1 \cap K \cap L$. Additionally, $N_G(H)$ has shape $(23 \sd 11) \times 2$ and $N_L(H) = N_G(H)$, while ${|N_G(H):N_{M_1}(H)|} = 22$. Since the $22$ non-identity elements of $H$ fall into two $K$-conjugacy classes and $C_K(H) = H$, we conclude that $N_K(H)$ has shape $23 \sd 11$, and so ${|N_G(H):N_K(H)|} = 2$.

Consider the pairs $(H',M')$, where $H'$ is a $G$-conjugate of $H$, $M'$ is a $G$-conjugate of $M_1$, and $H' \le M'$. As any two $G$-conjugates of $H$ appear in an equal number of such pairs, we deduce that $H$ lies in exactly ${|N_G(H):N_{M_1}(H)|} = 22$ $G$-conjugates of $M_1$. Similarly, $H$ lies in two $G$-conjugates of $K$ and one $G$-conjugate of $L$. %Additionally, the number of $G$-conjugates of $H$ in $M_1$, $L$ and $K$ are $|M_1:N_{M_1}(H)| = 47$, $a:=|L:N_L(H)|$ and $b:=|K:N_K(H)|$, respectively.

% As $|C_G(H)| = 46$ and $|Z(L)| = 2$, the centraliser in $L$ of each of its subgroups of order $23$ also has order $46$. However, each subgroup of $K$ of order $23$ is self-centralising. Since $\mathrm{Co}_2$ has two conjugacy classes of elements of order $23$, so does $L$. The same is true for $K$. Hence there are $b:=2|L|/(22 \cdot 46)$ and $c:={2|K|/(22 \cdot 23)}$ $G$-conjugates of $H$ in $L$ and $K$, respectively. Additionally, $H$ lies in ${b|G:L|/a} = 1$ $G$-conjugate of $L$, and $c|G:K|/a = 2$ $G$-conjugates of $K$.

As $M_1$ has shape $47 \sd 23$, it contains a subgroup $S$ of order $47$. In fact, $M_1$ is the unique maximal subgroup of $G$ that contains $S$. Hence if $J$ is a maximal subgroup of $G$ satisfying $J \ne M_1$ and $J \cap M_1 \ne 1$, then $J$ contains a $G$-conjugate of $H$. Let $\mathcal{U}$ be the set of $G$-conjugates of $H$ that lie in at least one such maximal subgroup $J$, or in $M_1$. There are $47$ subgroups of order $23$ in $M_1$, each of which lies in two $G$-conjugates of $K$, and there are $|K:N_K(H)|$ subgroups of order $23$ in $K$. Therefore, there are fewer than $47 \cdot 2|K:N_K(H)|$ subgroups in $\mathcal{U}$ that lie in at least one $G$-conjugate of $K$. By considering the $G$-conjugates of $M_1$ and $L$ similarly, we conclude that $$|\mathcal{U}| < 47({2|K:N_K(H)|}+ 22\cdot 47 + |L:N_L(H)|) < |G:M_1|/22.$$ Hence there exists $g \in G$ such that no subgroup of $M_1^g$ lies in $\mathcal{U}$. This means that $M_1$ and $M_1^g$ are not adjacent in $\ig$ and have no common neighbours, and so $d(M_1,M_1^g) > 2$. As $M_1$ and $M_1^g$ are the unique neighbours of $S$ and $S^g$, respectively, it follows that $d(S,S^g) > 4$. Therefore, $\diam(\ig) = 5$.

\medskip

\noindent (viii) $G = \mathbb{M}$. Liebeck and Saxl list two possible maximal subgroups of odd order (up to conjugacy), of shape $59 \sd 29$ and $71 \sd 35$, respectively. However, these subgroups are not, in fact, maximal: the former lies in the maximal subgroup $\mathrm{L}_2(59)$ constructed in \cite{holmes59}, and the latter lies in the maximal subgroup $\mathrm{L}_2(71)$ constructed in \cite{holmes71}. Hence $G$ has no maximal subgroup of odd order, and so $\diam(\ig) \le 4$.
%\end{itemize}
\end{proof}

\subsection*{Acknowledgment}
The author is grateful to Colva Roney-Dougal and Peter Cameron for proofreading this paper and providing helpful feedback, and to Peter for the linear group arguments used in the proof of Theorem \ref{thm:diamthm}.

\bibliographystyle{plain}
\bibliography{Intrefs}

\end{document}